\documentclass[12pt,centertags,oneside,reqno]{amsart}
\usepackage{amsmath,amstext,amsthm,amscd,typearea,hyperref}

\usepackage{pict2e}
\usepackage{xcolor}
\usepackage{needspace}

\usepackage{amssymb}
\usepackage[T1]{fontenc}
\usepackage[utf8]{inputenc}
\usepackage{lmodern}
\usepackage[margin=27mm]{geometry}
\usepackage{amsmath,amssymb,amsthm,mathtools}
\usepackage{mathrsfs}
\usepackage{microtype}
\usepackage{enumitem}
\usepackage{hyperref}
\hypersetup{colorlinks=true,linkcolor=blue,citecolor=blue,urlcolor=blue,
  pdftitle={A rational surface with discrete, non-finitely generated full automorphism group}}
\setlist[enumerate]{label=\textup{(\roman*)},leftmargin=*,itemsep=3pt}
\newtheorem{theorem}{Theorem}[section]
\newtheorem{question}[theorem]{Question}
\newtheorem{lemma}[theorem]{Lemma}
\newtheorem{proposition}[theorem]{Proposition}

\newtheorem{remark}[theorem]{Remark}
\numberwithin{equation}{section}
\newcommand{\CC}{\mathbb C}
\newcommand{\ZZ}{\mathbb Z}
\newcommand{\QQ}{\mathbb Q}

\newcommand{\PP}{\mathbb P}
\DeclareMathOperator{\Aut}{Aut}
\DeclareMathOperator{\Ine}{Ine}
\DeclareMathOperator{\Pic}{Pic}
\DeclareMathOperator{\NS}{NS}
\DeclareMathOperator{\Bl}{Bl}
\DeclareMathOperator{\Km}{Km}

\DeclareMathOperator{\id}{id}
\DeclareMathOperator{\Spec}{Spec}
\newcommand{\Ecal}{\mathcal E}
\newcommand{\Fcal}{\mathcal F}

\title[Non-finitely generated automorphism group]{A rational surface with discrete and non-finitely generated automorphism group}

\author{Tien-Cuong Dinh}
\address{Department of Mathematics, National University of Singapore, 10, 
Lower Kent Ridge Road, Singapore 119076}
\email{matdtc@nus.edu.sg}
\author{Keiji Oguiso}
\address{Mathematical Sciences, the University of Tokyo, Meguro Komaba 3-8-1,
Tokyo, Japan, and National Center for Theoretical Sciences,
Mathematics Division, National Taiwan University,
Taipei, Taiwan}
\email{oguiso@ms.u-tokyo.ac.jp}
\author{Xun Yu}
\address{Center for Applied Mathematics and KL-AAGDM, Tianjin University, Weijin Road 92, Tianjin 300072, P.R. China}
\email{xunyu@tju.edu.cn}
\author{Qi Zhou}
\address{Department of Mathematics, National University of Singapore, 10, 
Lower Kent Ridge Road, Singapore 119076}
\email{e1124859@u.nus.edu}

\begin{document}
\begin{abstract}
We construct a smooth complex projective rational surface whose
automorphism group is discrete and not finitely generated. It is obtained
by blowing up a point on a branch curve of a rational quotient of a
product Kummer surface. Every point with transcendental coordinate gives
such a surface. The proof uses leading jets along the branch curve, the complete family of blowups and the cubic intersection of the family, which are not being considered in relevant earlier works.
\end{abstract}
\maketitle

\section{Introduction}

We work over complex field $\CC$. The aim of this paper is to construct a smooth
projective rational surface whose automorphism group is discrete and not
finitely generated (Theorem~\ref{thm:main}).

In \cite{DO}, the first two authors constructed surfaces with these properties by
blowing up a K3 surface. In \cite{DOY}, the first three authors used a rational
quotient of a product Kummer surface to construct rational surfaces
with infinitely many mutually non-isomorphic real forms. Here we use the same rational quotient
but choose a different center for the final blowup.

We will use a fixed affine coordinate for $\PP^1$ which identifies it with $\CC\cup\{\infty\}$ and denote by $\Bl$ the blowup operator. Consider the rational surface
\[
 T:=\Bl_{\{0,1,2,\infty\}\times\{0,1,3,\infty\}}
          (\PP^1\times\PP^1),
\]
and let $C$ be the strict transform of $\PP^1\times\{\infty\}$.
The natural projection on the first factor $\PP^1$ identifies $C$ with $\PP^1$. Therefore, we can write $x$ for
the resulting coordinate on $C$. Observe that $C^2=-4$.
Denote by $\overline{\QQ}$ the field of algebraic numbers in $\CC$.

\begin{theorem}\label{thm:main}
With the above notation, let $q\in C$ be a point such that $x(q)\in\CC\setminus\overline{\QQ}$.
Then the smooth projective rational surface $X_q:=\Bl_qT$ satisfies
\[
 \Aut^0(X_q)=\{\id_{X_q}\} \qquad \text{and}\qquad
 \qquad \Aut(X_q)\ \text{is not finitely generated}.
\]
\end{theorem}

Here, $\Aut(X_q)$ is the group of automorphisms of $X_q$ and $\Aut^0(X_q)$ is its connected component of the identity map $\id_{X_q}$.

Note that the blowup center in Theorem~\ref{thm:main} lies on the branch curve $C$ while
in \cite{DOY}, the final center lies on an exceptional $(-1)$-curve
of the grid blowup. 

 Let $V$ be a smooth complex projective variety. For any subvariety $D\subset V$, we write
\[
 \begin{aligned}
 \Aut(V,D)&:=\{f\in\Aut(V)\mid f(D)=D\},\\
 \Ine(V,D)&:=\{f\in\Aut(V,D)\mid f|_D=\id_D\}.
 \end{aligned}
\]
We will use later superscript $s$ 
for $\Aut$ and $\Ine$ to 
denote similar groups for symplectic automorphisms of a K3 surface.
Note also that when considering the action on a Picard lattice, we use
$f\mapsto(f^{-1})^*$.

\medskip
\noindent\textbf{Outline of the proof.}
In Section~\ref{sec:kummer}, we recall three automorphisms of the
Kummer surface $S:=\Km(E\times F)$ that descend via the quotient map
$\pi:S\longrightarrow S/\langle\theta\rangle=T$ and induce on $C$ the maps
\[
 x\longmapsto 2x,\qquad x\longmapsto x+1,\qquad
 x\longmapsto 2/x.
\]
A result of the second author \cite{Og19} yields a nontrivial
symplectic automorphism fixing the relevant curve pointwise. Its first
nonzero jet in the normal direction belongs to an additive group of
polynomials. By taking finite differences and conjugating by
$x\mapsto 2x$, we produce a copy of $\ZZ[1/2]$ in the image of the jet
homomorphism. It follows that $\Ine(T,C)$ is not finitely generated.

We then consider the family $Y\to C$ whose fiber over $q$ is $X_q$.
The intersection cubic shows that every automorphism of $Y$ over $C$
preserves the exceptional divisor. Hence $\Aut_C(Y)=\Ine(T,C)$.
If the automorphism group of the geometric generic fiber were finitely
generated, each element of a finite set of generators and its inverse would
fail to extend over at most finitely many parameters. The resulting
finite set would be invariant under the three maps above, which is
impossible unless it is empty. Thus all automorphisms of the geometric
generic fiber would extend over $C$, contradicting the non-finite
generation of $\Ine(T,C)$. Finally, because the family is defined over
$\QQ$, a change of algebraically closed ground field identifies the
geometric generic automorphism group with $\Aut(X_q)$ whenever $x(q)$
is transcendental over $\QQ$.

The inertia group is studied in Section~\ref{sec:inertia}, and the
family $Y\to C$ in Section~\ref{sec:family}. We prove Theorem~\ref{thm:main}
in Section~\ref{sec:proof}.

Although $T$ is defined over $\overline{\QQ}$, the surfaces $X_q$
in Theorem~\ref{thm:main} do not admit models over $\overline{\QQ}$.
This raises the following question.

\begin{question}\label{quest:main2}
Does there exist a smooth projective rational surface defined over
$\overline{\QQ}$ whose automorphism group is discrete and not finitely
generated?
\end{question}

\medskip
\noindent\textbf{Brief history and current status.}
To the best of our knowledge, Barry Mazur~\cite[p.~42, \S17]{Ma}
was the first to raise the question of finite generation of the
component group of the automorphism group of a smooth projective variety:
\begin{question}\label{quest:main}
Is the group $\Aut(V)/\Aut^0(V)$ finitely generated for every smooth projective
variety $V$?
\end{question}
The answer is affirmative when $\dim V=1$ and for several classes of surfaces,  namely, Del Pezzo surfaces as their nef cones are finite rational polyhedral cones, K3 surfaces by Sterk~\cite{St85}, Enriques surfaces by Namikawa~\cite{Na85} and so on. See also Theorem~\ref{thm:current1} (iii) below. Lesieutre~\cite{Le} gave the first negative answer, constructing a smooth projective
sixfold $V$ whose automorphism group is discrete and not finitely
generated. His sixfold is uniruled but not rationally connected; its
construction starts from a rational surface $W$ with a discrete,
finitely generated automorphism group. In the same paper,
Lesieutre~\cite[Remark~3]{Le} suggested blowing up $W$ at a very
general point of a suitable curve to obtain a rational surface with
a discrete, non-finitely generated automorphism group, and explained
the two obstacles to this approach. This appears to be the first
explicit formulation of the question for rational surfaces.
The first two authors  subsequently constructed the first smooth
projective surface with a discrete, non-finitely generated
automorphism group by blowing up a K3 surface \cite{DO}. The results of
\cite{DO,DOY2,DGLOWY} give the following picture.

\begin{theorem}\label{thm:current1}
\begin{enumerate}
\item If $V$ is a smooth projective variety of dimension $d\ge 2$
such that $\Aut(V)/\Aut^0(V)$ is not finitely generated, then its
Kodaira dimension $\kappa(V)$ satisfies $\kappa(V)\le d-2$.
Conversely, for every pair $(d,\kappa)$ with $d\ge 3$ and
$\kappa\le d-2$, there exists a smooth projective variety $V$ of
dimension $d$ and Kodaira dimension $\kappa$ whose automorphism group
is discrete and not finitely generated.

\item For every $d\ge 3$, there exists a smooth projective rational
variety $V$ of dimension $d$ whose automorphism group is discrete and
not finitely generated.

\item If $V$ is a smooth projective surface such that
$\Aut(V)/\Aut^0(V)$ is not finitely generated, then $V$ is rational
or is a non-minimal surface birational to a K3 or Enriques surface.
Conversely, there exist non-minimal surfaces birational to K3 and
Enriques surfaces whose automorphism groups are discrete and not
finitely generated.
\end{enumerate}
\end{theorem}

Thus the rational surface case remained the principal gap in this
coarse classification.

\medskip
\noindent\textbf{Relation to the real form problem.}
The question of finite generation of $\Aut(V)/\Aut^0(V)$ has a close
parallel in the study of real forms. In particular, the results of
\cite{Le,DO,DOY,DGLOWY} yield the following statement, analogous to
Theorem~\ref{thm:current1}.

\begin{theorem}\label{thm:current2}
\begin{enumerate}
\item If a smooth projective variety $V$ of dimension $d\ge 2$ has
infinitely many real forms, then $\kappa(V)\le d-2$.
Conversely, for every pair $(d,\kappa)$ with $d\ge 3$ and
$\kappa\le d-2$, there exists a smooth projective variety $V$ of
dimension $d$ and Kodaira dimension $\kappa$ with infinitely many
mutually non-isomorphic real forms.

\item For every $d\ge 2$, there exists a smooth projective rational
variety $V$ of dimension $d$ with infinitely many mutually non-isomorphic real forms.

\item If a smooth projective surface $V$ has infinitely many real
forms, then $V$ is rational or is a non-minimal surface birational
to a K3 or Enriques surface. Conversely, examples with infinitely
many mutually non-isomorphic real forms exist in each of these three classes.
\end{enumerate}
\end{theorem}

In \cite{DOY}, we constructed a smooth projective rational surface
with infinitely many mutually non-isomorphic real forms by blowing up $T$ at a suitable real
point on the image of $C_{11}$. We note that such a surface can also be obtained by blowing up $T$ at a suitable transcendental real point
on $C$. There is a smooth projective rational surface satisfying the property on real forms together with the properties in Theorem \ref{thm:main}.

\medskip
\noindent\textbf{Acknowledgements.} 
The first author is supported by the NUS grant A-8003576-00-00. The
second author is supported by JSPS Grant 25H00587, 25K21992 and NCTS scholar program. The third author is supported by NSFC (No. 12071337). The fourth author is supported by the NUS President’s Graduate Fellowship.

\medskip
\noindent\textbf{AI disclosure.}
The authors developed the main ideas, technical framework,
construction strategy for the surfaces, and checked all the details. OpenAI's ChatGPT 6 Astra assisted in
the construction of concrete examples and some arguments used in the proof. The authors assume responsibility for these
contributions and for the manuscript as a whole.


\section{Preliminaries}\label{sec:kummer}

In this section, we will recall several notations and properties from 
\cite{DOY} concerning a Kummer surface ${\rm Km}(E\times F)$, its rational quotient surface $T$, and their automorphisms. New statements presented here will be given proofs.

Let $E$ and $F$ be the elliptic curves defined respectively by the Weierstrass forms
\[
 E:y^2=x(x-1)(x-2)\qquad \text{ and} \qquad 
 F:(y')^2=x'(x'-1)(x'-3).
\]
The curves $E$ and $F$ are not isogenous. Indeed, $j(E)=1728$ and $j(F)=21952/9$. The first curve has complex
multiplication, whereas the second does not, since a CM $j$-invariant
is an algebraic integer. Isogeny preserves the rational endomorphism
algebra. Consider the Kummer K3 surface 
$$S:=\Km(E\times F)$$
associated to the abelian surface $E\times F$.

The $2$-torsion groups of $E$ and $F$ are denoted by $\{b_i\}_{i=1}^4$ and $\{a_i\}_{i=1}^4$ respectively. Then $S$ contains 24 smooth rational curves, as shown in Figure \ref{fig1}. These consist of eight smooth rational curves $E_i$ and $F_i$ ($1\le i\le 4$), which arise from the eight elliptic curves $E\times\{a_i\}$ and $\{b_i\}\times F$ on $E\times F$, together with sixteen exceptional curves $C_{ij}$ ($1\le i,j\le 4$) lying over the sixteen singular points of type $A_1$ on the quotient surface $E\times F/\langle -1_{E\times F}\rangle$. The configuration of these 24 smooth rational curves on $S$ is also displayed in Figure \ref{fig1}.

\begin{figure}
\centering
\unitlength 0.1in
\begin{picture}(25,24)(-1,-23.5)

\put(4.5,-22){\makebox(0,0)[rb]{$F_1$}}
\put(9.5,-22){\makebox(0,0)[rb]{$F_2$}}
\put(14.5,-22){\makebox(0,0)[rb]{$F_3$}}
\put(19.5,-22){\makebox(0,0)[rb]{$F_4$}}
\put(0.25,-18.5){\makebox(0,0)[lb]{$E_1$}}
\put(0.25,-13.5){\makebox(0,0)[lb]{$E_2$}}
\put(0.25,-8.5){\makebox(0,0)[lb]{$E_3$}}
\put(0.25,-3.5){\makebox(0,0)[lb]{$E_4$}}
\put(6,-16){\makebox(0,0)[lt]{$C_{11}$}}
\put(6,-11){\makebox(0,0)[lt]{$C_{12}$}}
\put(6,-6){\makebox(0,0)[lt]{$C_{13}$}}
\put(6,-1){\makebox(0,0)[lt]{$C_{14}$}}
\put(11,-16){\makebox(0,0)[lt]{$C_{21}$}}
\put(11,-11){\makebox(0,0)[lt]{$C_{22}$}}
\put(11,-6){\makebox(0,0)[lt]{$C_{23}$}}
\put(11,-1){\makebox(0,0)[lt]{$C_{24}$}}
\put(16,-16){\makebox(0,0)[lt]{$C_{31}$}}
\put(16,-11){\makebox(0,0)[lt]{$C_{32}$}}
\put(16,-6){\makebox(0,0)[lt]{$C_{33}$}}
\put(16,-1){\makebox(0,0)[lt]{$C_{34}$}}
\put(21,-16){\makebox(0,0)[lt]{$C_{41}$}}
\put(21,-11){\makebox(0,0)[lt]{$C_{42}$}}
\put(21,-6){\makebox(0,0)[lt]{$C_{43}$}}
\put(21,-1){\makebox(0,0)[lt]{$C_{44}$}}

\polyline(5,0)(5,-22)
\polyline(10,0)(10,-22)
\polyline(15,0)(15,-22)
\polyline(20,0)(20,-22)

\polyline(0,-19)(4.5,-19)\polyline(5.5,-19)(9.5,-19)
\polyline(10.5,-19)(14.5,-19)\polyline(15.5,-19)(19.5,-19)

\polyline(0,-14)(4.5,-14)\polyline(5.5,-14)(9.5,-14)
\polyline(10.5,-14)(14.5,-14)\polyline(15.5,-14)(19.5,-14)

\polyline(0,-9)(4.5,-9)\polyline(5.5,-9)(9.5,-9)
\polyline(10.5,-9)(14.5,-9)\polyline(15.5,-9)(19.5,-9)

\polyline(0,-4)(4.5,-4)\polyline(5.5,-4)(9.5,-4)
\polyline(10.5,-4)(14.5,-4)\polyline(15.5,-4)(19.5,-4)

\polyline(2,-20)(6,-16)
\polyline(2,-15)(6,-11)
\polyline(2,-10)(6,-6)
\polyline(2,-5)(6,-1)
\polyline(7,-20)(11,-16)
\polyline(7,-15)(11,-11)
\polyline(7,-10)(11,-6)
\polyline(7,-5)(11,-1)
\polyline(12,-20)(16,-16)
\polyline(12,-15)(16,-11)
\polyline(12,-10)(16,-6)
\polyline(12,-5)(16,-1)
\polyline(17,-20)(21,-16)
\polyline(17,-15)(21,-11)
\polyline(17,-10)(21,-6)
\polyline(17,-5)(21,-1)

\end{picture}
\caption{Curves $E_i$, $F_j$ and $C_{ij}$ with
  $E_i^2=F_j^2=C_{ij}^2=-2$}
\label{fig1}
\end{figure}
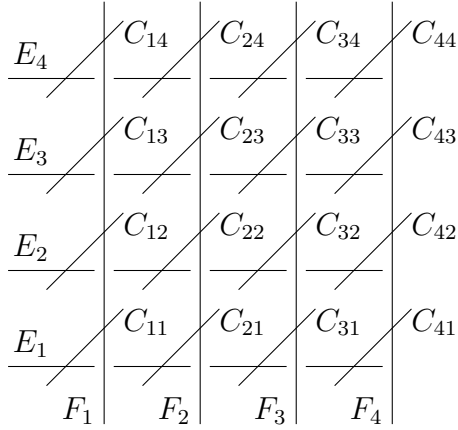

Throughout this paper, we use $x$ for the affine coordinate of $E_1 = E/\langle -1_E \rangle$ and set (with respect to this coordinate):
$$ E_{1}\,\, ,\,\, P := E_1 \cap C_{11} = \infty\,\, ,\,\, E_1 \cap C_{21} = 0\,\, ,\,\, E_1 \cap C_{31} = 1\,\, ,\,\, E_1 \cap C_{41} = 2.$$

The involution $\theta\in{\rm Aut}(S)$ induced by $(1_E,-1_F)\in {\rm Aut}(E\times F)$ fixes $E_1$ pointwise.
The associated quotient map is
\[
 \pi:S\longrightarrow S/\langle\theta\rangle=T,
\]
and $\pi|_{E_1}:E_1\to C$ is an isomorphism (note that in \cite{DOY}  $C$ denotes the curve $E_1$, but in this paper, we use $C$ to denote the image of $E_1$ under $\pi$).
By
\cite[Lemma~2.1]{DOY}, $\theta$ is central in $\Aut(S)$ and
\begin{equation}\label{eq:quotient}
 \Aut(T)\simeq\Aut(S)/\langle\theta\rangle .
\end{equation}
See also \cite[Lemma~3.3]{DO} for centrality. 

Next, we recall one involution $\iota$ of $S$ and 
its conjugates.
As in \cite[Page 956]{DO} and \cite[Page 271]{DOY}, consider the elliptic fibrations $\Phi_{D_i} : S \to \mathbb{P}^1$ ($i = 1$, $2$) on $S$ defined respectively by the complete linear systems $|D_i|$ of the divisors of Kodaira's singular fiber type,
$$D_1 := E_1 + C_{11} + F_1 + C_{12} + E_2 + C_{22} + F_2 + C_{21}$$
and
$$D_2 := E_1 + 2C_{11} + E_2 + 2C_{12} + E_{3} + 2C_{13} + 3F_{1},$$
see Figure \ref{fig2}. We choose $C_{31}$ as the zero section of both $\Phi_{D_1}$ and $\Phi_{D_2}$. 

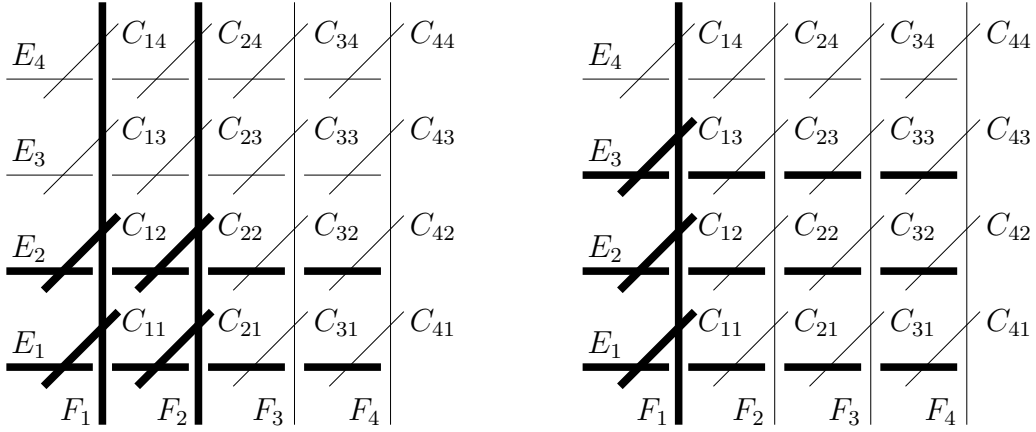
\begin{figure}
\centering
\unitlength 0.1in
\begin{picture}(56,25)(-16,-24)


\put(-10.5,-22){\makebox(0,0)[rb]{$F_1$}}
\put(-5.5,-22){\makebox(0,0)[rb]{$F_2$}}
\put(-0.5,-22){\makebox(0,0)[rb]{$F_3$}}
\put(4.5,-22){\makebox(0,0)[rb]{$F_4$}}
\put(-14.75,-18.5){\makebox(0,0)[lb]{$E_1$}}
\put(-14.75,-13.5){\makebox(0,0)[lb]{$E_2$}}
\put(-14.75,-8.5){\makebox(0,0)[lb]{$E_3$}}
\put(-14.75,-3.5){\makebox(0,0)[lb]{$E_4$}}
\put(-9,-16){\makebox(0,0)[lt]{$C_{11}$}}
\put(-9,-11){\makebox(0,0)[lt]{$C_{12}$}}
\put(-9,-6){\makebox(0,0)[lt]{$C_{13}$}}
\put(-9,-1){\makebox(0,0)[lt]{$C_{14}$}}
\put(-4,-16){\makebox(0,0)[lt]{$C_{21}$}}
\put(-4,-11){\makebox(0,0)[lt]{$C_{22}$}}
\put(-4,-6){\makebox(0,0)[lt]{$C_{23}$}}
\put(-4,-1){\makebox(0,0)[lt]{$C_{24}$}}
\put(1,-16){\makebox(0,0)[lt]{$C_{31}$}}
\put(1,-11){\makebox(0,0)[lt]{$C_{32}$}}
\put(1,-6){\makebox(0,0)[lt]{$C_{33}$}}
\put(1,-1){\makebox(0,0)[lt]{$C_{34}$}}
\put(6,-16){\makebox(0,0)[lt]{$C_{41}$}}
\put(6,-11){\makebox(0,0)[lt]{$C_{42}$}}
\put(6,-6){\makebox(0,0)[lt]{$C_{43}$}}
\put(6,-1){\makebox(0,0)[lt]{$C_{44}$}}

\linethickness{1mm}
\polyline(-10,0)(-10,-22)
\polyline(-5,0)(-5,-22)
\linethickness{0.1mm}
\polyline(0,0)(0,-22)
\polyline(5,0)(5,-22)

\linethickness{1mm}
\polyline(-15,-19)(-10.5,-19)\polyline(-9.5,-19)(-5.5,-19)
\polyline(-4.5,-19)(-0.5,-19)\polyline(0.5,-19)(4.5,-19)
\polyline(-15,-14)(-10.5,-14)\polyline(-9.5,-14)(-5.5,-14)
\polyline(-4.5,-14)(-0.5,-14)\polyline(0.5,-14)(4.5,-14)
\linethickness{0.1mm}
\polyline(-15,-9)(-10.5,-9)\polyline(-9.5,-9)(-5.5,-9)
\polyline(-4.5,-9)(-0.5,-9)\polyline(0.5,-9)(4.5,-9)
\polyline(-15,-4)(-10.5,-4)\polyline(-9.5,-4)(-5.5,-4)
\polyline(-4.5,-4)(-0.5,-4)\polyline(0.5,-4)(4.5,-4)

\linethickness{3pt}
\polyline(-13.06,-20)(-9.17,-16.11)   
\polyline(-13.06,-15)(-9.17,-11.11)   
\polyline(-8.19,-20)(-4.31,-16.11)    
\polyline(-8.19,-15)(-4.31,-11.11)    
\linethickness{0.33pt}
\polyline(-13.06,-10)(-9.17,-6.11)    
\polyline(-13.06,-5)(-9.17,-1.11)     
\polyline(-8.19,-10)(-4.31,-6.11)     
\polyline(-8.19,-5)(-4.31,-1.11)      
\polyline(-3.19,-20)(0.69,-16.11)     
\polyline(-3.19,-15)(0.69,-11.11)     
\polyline(-3.19,-10)(0.69,-6.11)      
\polyline(-3.19,-5)(0.69,-1.11)       
\polyline(1.81,-20)(5.69,-16.11)      
\polyline(1.81,-15)(5.69,-11.11)      
\polyline(1.81,-10)(5.69,-6.11)       
\polyline(1.81,-5)(5.69,-1.11)        


\put(19.5,-22){\makebox(0,0)[rb]{$F_1$}}
\put(24.5,-22){\makebox(0,0)[rb]{$F_2$}}
\put(29.5,-22){\makebox(0,0)[rb]{$F_3$}}
\put(34.5,-22){\makebox(0,0)[rb]{$F_4$}}
\put(15.25,-18.5){\makebox(0,0)[lb]{$E_1$}}
\put(15.25,-13.5){\makebox(0,0)[lb]{$E_2$}}
\put(15.25,-8.5){\makebox(0,0)[lb]{$E_3$}}
\put(15.25,-3.5){\makebox(0,0)[lb]{$E_4$}}
\put(21,-16){\makebox(0,0)[lt]{$C_{11}$}}
\put(21,-11){\makebox(0,0)[lt]{$C_{12}$}}
\put(21,-6){\makebox(0,0)[lt]{$C_{13}$}}
\put(21,-1){\makebox(0,0)[lt]{$C_{14}$}}
\put(26,-16){\makebox(0,0)[lt]{$C_{21}$}}
\put(26,-11){\makebox(0,0)[lt]{$C_{22}$}}
\put(26,-6){\makebox(0,0)[lt]{$C_{23}$}}
\put(26,-1){\makebox(0,0)[lt]{$C_{24}$}}
\put(31,-16){\makebox(0,0)[lt]{$C_{31}$}}
\put(31,-11){\makebox(0,0)[lt]{$C_{32}$}}
\put(31,-6){\makebox(0,0)[lt]{$C_{33}$}}
\put(31,-1){\makebox(0,0)[lt]{$C_{34}$}}
\put(36,-16){\makebox(0,0)[lt]{$C_{41}$}}
\put(36,-11){\makebox(0,0)[lt]{$C_{42}$}}
\put(36,-6){\makebox(0,0)[lt]{$C_{43}$}}
\put(36,-1){\makebox(0,0)[lt]{$C_{44}$}}

\linethickness{1mm}
\polyline(20,0)(20,-22)
\linethickness{0.1mm}
\polyline(25,0)(25,-22)
\polyline(30,0)(30,-22)
\polyline(35,0)(35,-22)

\linethickness{1mm}
\polyline(15,-19)(19.5,-19)\polyline(20.5,-19)(24.5,-19)
\polyline(25.5,-19)(29.5,-19)\polyline(30.5,-19)(34.5,-19)
\polyline(15,-14)(19.5,-14)\polyline(20.5,-14)(24.5,-14)
\polyline(25.5,-14)(29.5,-14)\polyline(30.5,-14)(34.5,-14)
\polyline(15,-9)(19.5,-9)\polyline(20.5,-9)(24.5,-9)
\polyline(25.5,-9)(29.5,-9)\polyline(30.5,-9)(34.5,-9)
\linethickness{0.1mm}
\polyline(15,-4)(19.5,-4)\polyline(20.5,-4)(24.5,-4)
\polyline(25.5,-4)(29.5,-4)\polyline(30.5,-4)(34.5,-4)

\linethickness{3pt}
\polyline(16.94,-20)(20.83,-16.11)    
\polyline(16.94,-15)(20.83,-11.11)    
\polyline(16.94,-10)(20.83,-6.11)     
\linethickness{0.33pt}
\polyline(16.94,-5)(20.83,-1.11)      
\polyline(21.81,-20)(25.69,-16.11)    
\polyline(21.81,-15)(25.69,-11.11)    
\polyline(21.81,-10)(25.69,-6.11)     
\polyline(21.81,-5)(25.69,-1.11)      
\polyline(26.81,-20)(30.69,-16.11)    
\polyline(26.81,-15)(30.69,-11.11)    
\polyline(26.81,-10)(30.69,-6.11)     
\polyline(26.81,-5)(30.69,-1.11)      
\polyline(31.81,-20)(35.69,-16.11)    
\polyline(31.81,-15)(35.69,-11.11)    
\polyline(31.81,-10)(35.69,-6.11)     
\polyline(31.81,-5)(35.69,-1.11)      

\end{picture}
\caption{Divisors $D_1$ and $D_2$.}
\label{fig2}
\end{figure}

Let \(\iota\) denote the inversion of the elliptic fibration \(\Phi_{D_2}\) with respect to the zero section \(C_{31}\). Then \(\iota\in \Aut(S)\), and \(\iota\) is an involution of \(S\) satisfying $\iota^*\omega_S=-\omega_S$. 
Here we fix a non-zero global holomorphic 2-form $\omega_S$
on $S$, which is unique up to scalar multiplications by $\CC\setminus\{0\}$.

Let \(f_1\) denote the translation automorphism of the elliptic fibration \(\Phi_{D_1}\) determined by the section \(C_{41}\) with respect to the zero section \(C_{31}\). Then \(f_1\in \Aut(S)\) and $f_1^*\omega_S=\omega_S$. Finally, for every \(n\in\mathbb{Z}\), define
\[
\iota_n := f_1^{-n}\circ \iota \circ f_1^{n}.
\]

We need one involution in $\Aut(S)$ not considered in \cite{DOY}. Let $t_0\in {\rm Aut}(E)$ denote the translation by the $2$-torsion point  $(1,0)\in E$ with respect to the origin $(2, 0) \in E$ under the affine coordinates $(x, y)$ in the equation of $E$. Then $t_0\times 1_F\in {\rm Aut}(E\times F)$ induces an involution $\phi\in {\rm Aut}(S)$ satisfying  $\phi(E_1)=E_1$ and $\phi^*(\omega_S)=\omega_S$. The following lemma will be used in Section \ref{sec:proof}.

\begin{lemma}\label{lem:maps}
The automorphisms $f_1,\iota_0\iota_1, \phi\in\Aut^s(S,E_1)$ satisfy
\begin{equation}\label{eq:maps}
 f_1|_{E_1}(x)=2x,\qquad (\iota_0\iota_1)|_{E_1}(x)=x+1,\qquad \phi |_{E_1}(x)=2/x.
\end{equation}
They descend to automorphisms of $T$ preserving $C$. Moreover, their action on
$C$ has no nonempty finite invariant subset.
\end{lemma}

\begin{proof}
By \cite[Lemma~2.2(3)]{DOY}, $f_1|_{E_1}(x)=2x$, and the
anti-symplectic involutions $\iota_n$ satisfy
 $\iota_n|_{E_1}(x)=2^{1-n}-x$. Thus $\iota_0\iota_1$ is symplectic
and restricts to $x\mapsto x+1$.

By direct calculation, translation of $E$ by the $2$-torsion point $(1,0)\in E$ with respect to the origin $(2, 0) \in E$ is given by
\[
 (x,y)\longmapsto\left(\frac2x,-\frac{2y}{x^2}\right).
\]
This implies $\phi |_{E_1}(x)=2/x$.

The second assertion follows from \eqref{eq:quotient}. For the last assertion, it is enough to observe that the only non-empty
finite subset of $C$ invariant under $x\mapsto x+1$ is
$\{\infty\}$ which is not invariant by $x\mapsto2/x$.
\end{proof}

\begin{lemma}\label{lem:faithful}
The actions of $\Aut(T)$ on $\Pic(T)$ and of $\Aut(X_q)$ on
$\Pic(X_q)$ are faithful. The same statements hold over every
algebraically closed extension of $\CC$, for every center $q$ on the
base-changed curve $C$. The action of $\Aut^s(S)$ on $\NS(S)$ is
faithful as well.
\end{lemma}

\begin{proof}
Recall that an irreducible curve $\Gamma$ of negative self-intersection is the unique effective divisor in its numerical class. Indeed, if $D$ is another effective divisor in this class, then $D-\Gamma$ is numerically zero. So its intersection with an ample class is zero and then the coefficient of $\Gamma$ in $D-\Gamma$ is negative. This contradicts the fact that its intersection with $\Gamma$ is also zero.

We say that an automorphism is Picard-trivial if its action on the Picard group is trivial. 
A Picard-trivial automorphism of $T$ therefore preserves the sixteen
exceptional curves. It descends to an automorphism of
$\PP^1\times\PP^1$ preserving the two ruling classes and fixing each
grid point. Both factors fix at least three points, so this
automorphism is the identity. 

Likewise, a Picard-trivial automorphism
of $X_q$ preserves the last exceptional curve and descends to the
identity on $T$. This argument works over any algebraically closed
extension and for every center. 

If $f\in\Aut^s(S)$ acts trivially on $\NS(S)$, its descent acts
trivially on $\Pic(T)\otimes\QQ$, since $\pi^*$ is injective.
The group $\Pic(T)$ is free, so the descent is the identity.
Consequently $f\in\langle\theta\rangle$. As $\theta$ is
anti-symplectic, we have $f=\id_S$.
\end{proof}

\section{The inertia group}\label{sec:inertia}
In this section, we show that the group $\Ine(T,C)$ is not finitely generated (Proposition \ref{prop:inertia}), which will play an important role in the proof of Theorem \ref{thm:main} in Section \ref{sec:proof}. 

We retain the notation of Section~\ref{sec:kummer}. Let $R:=E_1\subset S$. Then $C=\pi(R)\subset T$. Consider the group
\[
 K:=\Ine^s(S,R)
   =\{h\in\Aut(S,R)\mid h|_R=\id_R,\ h^*\omega_S=\omega_S\}.
\]
Roughly speaking, we will first prove that $K$ is not finitely generated, and then
compare it with $\Ine(T,C)$. Since every element of $K$ acts trivially
on $R$, restriction to $R$ gives no information about this group.
Instead, we consider the first nonzero term of its action in the
direction transverse to $R$. This gives a homomorphism from $K$ to
an additive group of sections of a line bundle on $R$ (see Lemma \ref{lem:homJ}).

\begin{lemma}\label{lem:Knontrivial}
The following statements hold:
\begin{enumerate}
    \item  The group $K$ contains a nonidentity element.
    \item  For any $h\in K$ and any $p\in R$, the differential $dh_p: T_{S,p}\rightarrow T_{S,p}$ is the identity.
\end{enumerate}
\end{lemma}

\begin{proof}
  By \cite[Proposition 4.1]{Og19}, the inertia group ${\rm Ine}(S,E_4)$ contains an element $f$ of infinite order. Consider the translation $t'$ of $F$ taking the $2$-torsion point $a_1$ to $a_4$. Then $1_E\times t'$ induces an automorphism $\phi'\in {\rm Aut}(S)$ satisfying $\phi'(E_1)=E_4$. Then $\phi'^{-1}f^N \phi'$ is a nonidentity element in $K$ for some $N>0$ (see \cite[Theorem 14.10]{Ue75}). This proves (i).

Fix $p\in R$. Since $\theta\in {\rm Aut}(S)$ (resp. $h\in K$) fixes $R$ pointwise and is
anti-symplectic (resp. symplectic), the differential $d\theta_p\in {\rm GL}(T_{S,p})$ (resp. $d h_p\in {\rm GL}(T_{S,p})$) has eigenvalues $1,-1$ (resp. $1,1$). Since $\theta$ is in the center of ${\rm Aut}(S)$, $d h_p$ and $d \theta_p$ commute. This implies that $d h_p$ is the identity on the tangent space
$T_{S,p}$. Thus  (ii) holds. \end{proof}

In particular, in any local coordinates $(z,w)$ on $S$ with $R=\{w=0\}$,
both components of the local holomorphic map $h(z,w)-(z,w)$ are divisible by $w^2$. Recall that $\omega_S$ is a nonzero global holomorphic $2$-form on $S$. There are local coordinates $(z,w)$ on $S$ such that $R=\{w=0\}$ and
$\omega_S=dz\wedge dw$. Indeed, start with local coordinates $(z,v)$ satisfying $R=\{v=0\}$. Then $\omega_S=f(z,v)\,dz\wedge dv$ for some holomorphic function $f$, and we get desired local coordinates $(z,w)$ by replacing $v$ by
$w=\int_0^v f(z,t)\,dt$. In the sequel, we use $O(w^r)$ to denote a holomorphic function divisible
by $w^r$.

\begin{lemma}\label{lem:m}
    There is an integer $m\ge2$ such that the following two statements hold.
    \begin{enumerate}
        \item Every
$h\in K$ has the local form
\begin{equation}\label{eq:jet}
 h(z,w)=\bigl(z+A_h(z)w^m+O(w^{m+1}),\
                   w+O(w^{m+1})\bigr)
\end{equation} in coordinates satisfying $R=\{w=0\}$ and
$\omega_S=dz\wedge dw$;
        \item For at least one $h\in K$, the coefficient $A_h$ is not identically zero.
    \end{enumerate}
  
\end{lemma}

\begin{proof}
    Let $\mathcal I_R$ be the ideal sheaf of $R$ in $S$. For a nonidentity
element $h\in K$, define its order of displacement along $R$ to be
the largest integer $r$ for which
\[
 h^*u-u\in\mathcal I_R^r
\]
for every local holomorphic function $u$ near $R$. Equivalently,
in local coordinates both components of $h-\id$ vanish to order at
least $r$ in $w$. The definition of the ideal sheaf shows that this condition
is independent of the coordinates. This order is finite: if all
orders vanished, the convergent local power series would give
$h=\id$ on a neighborhood of $R$, hence on the connected surface $S$. From this and Lemma \ref{lem:Knontrivial}, we have $m\ge2$. Every element of
$K$ is consequently the identity modulo $\mathcal I_R^m$, and some
element is not the identity modulo $\mathcal I_R^{m+1}$.

Choosing local coordinates $(z,w)$ with $R=\{w=0\}$ and
$\omega_S=dz\wedge dw$, we write
\[
 h(z,w)=\bigl(z+A_h(z)w^m+O(w^{m+1}),\
                   w+B_h(z)w^m+O(w^{m+1})\bigr).
\]
Since $h^*\omega_S=\omega_S$, the Jacobian determinant is identically
one. Its term of order $m-1$ in $w$ is particularly simple:
\[
 \det Dh(z,w)=1+mB_h(z)w^{m-1}+O(w^m).
\]
Thus $B_h=0$. This proves \eqref{eq:jet}. By the choice of $m$,
the tangential coefficient $A_h$ cannot vanish for every $h$.
\end{proof}

\begin{lemma}\label{lem:homJ}
    The leading coefficients $A_h$ in \eqref{eq:jet}
define a nonzero homomorphism
\begin{equation}\label{eq:jetmap}
 J:K\longrightarrow
 H^0(R,T_R\otimes N_{R/S}^{-m})
 \simeq H^0(\PP^1,T_{\PP^1}^{\otimes(m+1)}),
\end{equation}
where the target is regarded as an additive group. This
homomorphism is equivariant for conjugation by symplectic
automorphisms of $S$ preserving $R$.
\end{lemma}

\begin{proof}
    We first check additivity. If $h,k\in K$, substitution of their
expansions in \eqref{eq:jet} gives
\[
 (h\circ k)(z,w)
  =\bigl(z+(A_h(z)+A_k(z))w^m+O(w^{m+1}),\
          w+O(w^{m+1})\bigr).
\]
Terms coming from substituting $z+O(w^m)$ into $A_h(z)$ have order
at least $2m\ge m+1$ and do not affect this coefficient. Therefore
\[
 A_{h\circ k}=A_h+A_k,\qquad A_{h^{-1}}=-A_h.
\]

We next check that the coefficient has a global meaning. For two
adapted symplectic coordinate systems, the change of coordinates
along $R$ has the form
\[
 z'=\varphi(z)+O(w),\qquad
 w'=\frac{w}{\varphi'(z)}+O(w^2).
\]
The second formula follows by restricting
$dz'\wedge dw'=dz\wedge dw$ to $w=0$. Using \eqref{eq:jet}, the
leading tangential displacement in the new coordinates is
\[
 \varphi'(z)A_h(z)w^m
   =\varphi'(z)^{m+1}A_h(z)(w')^m
      \pmod{(w')^{m+1}}.
\]
Terms involving the normal displacement have higher order.
Consequently the local expressions
$A_h(z)(\partial/\partial z)\otimes(w^m)$ glue as sections of
$T_R\otimes(\mathcal I_R/\mathcal I_R^2)^{\otimes m}
=T_R\otimes N_{R/S}^{-m}$.

The symplectic form identifies $N_{R/S}$ with $T_R^*=K_R$:
a normal vector defines a covector on $T_R$ by pairing with
$\omega_S$. This explains the last isomorphism in \eqref{eq:jetmap}.
By Lemma \ref{lem:m} (ii), $J$ is nonzero.

The same calculation applies to conjugation. If
$g\in\Aut^s(S,R)$ and $\psi=g|_R$, then the coefficient of
$ghg^{-1}$, expressed in the coordinate $x$ on $R$, is
\begin{equation}\label{eq:jetconjugation}
 A_{ghg^{-1}}(x)
   =\bigl(\psi'(\psi^{-1}(x))\bigr)^{m+1}
       A_h(\psi^{-1}(x)).
\end{equation}
One first evaluates at the source point $\psi^{-1}(x)$; the
tangential differential contributes one factor of $\psi'$, and
the inverse normal differential contributes the other $m$ factors.
This proves the lemma.
\end{proof}

\begin{lemma}\label{lem:JK}
    The additive group $J(K)$, and hence $K$, is not finitely generated.
\end{lemma}

\begin{proof}
Put $k:=m+1$. In the affine coordinate $x$, a global section of
$T_{\PP^1}^{\otimes k}$ is written
$A(x)(\partial/\partial x)^{\otimes k}$. At infinity, use $u=1/x$.
The same section is
\[
 (-1)^k u^{2k}A(1/u)(\partial/\partial u)^{\otimes k}.
\]
Regularity at infinity is equivalent to $A$ being a polynomial of
degree at most $2k=2m+2$. Thus we identify the image
$\mathcal A:=J(K)$ with a subgroup of this polynomial space.

Recall the automorphisms $f_1$, $\iota_0\iota_1$ in Lemma \ref{lem:maps}. To simplify notation, put $a:=f_1$ and $b:=\iota_0\iota_1$. The automorphisms $a$ and $b$ normalize $K$, since both preserve $R$
and are symplectic. Applying \eqref{eq:jetconjugation} to their
restrictions $2x$ and $x+1$ gives
\begin{equation}\label{eq:jetactions}
 J(aha^{-1})(x)=2^{m+1}A_h(x/2),\qquad
 J(bhb^{-1})(x)=A_h(x-1).
\end{equation}
In particular, for a polynomial $A(x)\in \mathcal{A}$ of degree $d\ge 1$, the polynomial $A(x-1)-A(x)\in\mathcal{A}$ is nonzero and of degree $d-1$. From this and Lemma \ref{lem:homJ}, we infer that $\mathcal A$ contains a nonzero constant polynomial, say $\alpha$.

Conjugation by $a^{-1}$ sends $A(x)$ to
$2^{-(m+1)}A(2x)$. Repeatedly applying this to $\alpha$ yields
\[
 \frac{\alpha}{2^{n(m+1)}}\in\mathcal A\qquad(n\ge0).
\]
For any integer $r\ge0$, choose $n$ with $n(m+1)\ge r$ and multiply
this element by the integer $2^{n(m+1)-r}$. It follows that
\[
 \alpha\ZZ[1/2]\subset\mathcal A,\,\,\text{ where }\,\,
 \ZZ[1/2]=\{l/2^r\mid l\in\ZZ,\ r\ge0\}.
\]
Thus by \cite[Proposition 2.5]{DO}, $\mathcal A$ and $K$ are not finitely generated.
\end{proof}
Now we are ready to prove the desired property for $\Ine(T,C)$.
\begin{proposition}\label{prop:inertia}
The group $\Ine(T,C)$ is not finitely generated.
\end{proposition}

\begin{proof}
    Since $S$ is a projective K3 surface, the canonical representation of its automorphism group is finite (see \cite[Theorem 14.10]{Ue75}). Thus $K$ is a subgroup of $\Ine(S,R)$ of finite index. Since $\pi|_R:R\to C$ is an isomorphism and
$\pi^{-1}(C)_{\mathrm{red}}=R$, the quotient homomorphism in
\eqref{eq:quotient} restricts to an exact sequence
\[
 1\longrightarrow\langle\theta\rangle
  \longrightarrow \Ine(S,R)\longrightarrow\Ine(T,C)\longrightarrow1.
\] Then $K$ is isomorphic to a subgroup of $\Ine(T,C)$ of finite index. Thus by Lemma \ref{lem:JK} and \cite[Proposition 2.5]{DO}, $\Ine(T,C)$ is not finitely generated.\end{proof}

\section{The complete family}\label{sec:family}

In this section, we study automorphisms of the complete family arising from letting the blowup center vary over the whole curve $C$ contained in the surface $T$. Let $i:C\hookrightarrow T$ be the inclusion map. Consider
\[
 Z:=C\times T,\qquad
 \Gamma_C:=\{(q,i(q))\mid q\in C\},\qquad
 \beta:Y:=\Bl_{\Gamma_C}Z\longrightarrow Z.
\]
We use ${\rm pr}_1$ and ${\rm pr}_2$ to denote the projections from $Z$ to $C$ and $T$ respectively. The first projection induces a smooth projective morphism
\[
 p:=\operatorname{pr}_1\circ\beta:Y\longrightarrow C,
 \qquad p^{-1}(q)=X_q:=\Bl_qT.
\]

Let $\Ecal$ be the exceptional divisor of $\beta$, and let $\Fcal$
denote the divisor class of a fiber of $p$. For $L\in\Pic(T)$, we
also write $L$ for $\beta^*\operatorname{pr}_2^*L$ on $Y$ and for
its restriction to a fiber. The product and blowup formulas for
Picard groups give
\begin{equation}\label{eq:pic}
 \Pic(Y)=\Pic(T)\oplus\ZZ\Ecal\oplus\ZZ\Fcal,\qquad
 \Pic(X_q)=\Pic(T)\oplus\ZZ e_q,
\end{equation}
where $e_q=\Ecal|_{X_q}$ is the class of the exceptional curve
$E_q:=\Ecal\cap X_q$.

We will prove Lemma \ref{lem:relative} which concerns automorphisms of the entire family over the fixed base $C$. The main point is to recognize the prime divisor $\Ecal$ from
intersection numbers on the threefold $Y$ (Lemma \ref{lem:Eunique}). Once $\Ecal$ is preserved,
we can descend an automorphism to $C\times T$ and get an element in $\Ine(T,C)$.

\begin{lemma}
The cubic intersection form on $\Pic(Y)$ is given by
\begin{equation}\label{eq:cubic}
 (L+s\Ecal+t\Fcal)^3
    =3t(L^2-s^2)-s^2(3L\cdot C-2s).
\end{equation}
\end{lemma}

\begin{proof}
     Note that the normal bundle of the graph $\Gamma_C\subset Z$ is
$N_{\Gamma_C/Z}\simeq T_T|_C$. By adjunction and $C^2=-4$, we have $\deg N_{\Gamma_C/Z}=-K_T\cdot C=C^2+2=-2.$ For the blowup of a smooth curve in a smooth threefold, the
exceptional divisor satisfies
\[
 \beta_*(\Ecal^2)=-[\Gamma_C],\qquad
 \Ecal^3=-\deg N_{\Gamma_C/Z}.
\]
The projection formula now yields
\begin{equation}\label{eq:familyintersections}
 L\Ecal^2=-L\cdot C,\qquad
 \Fcal\Ecal^2=-1,\qquad
 \Ecal^3=2.
\end{equation}
For $L,L'\in\Pic(T)$ the remaining products are $LL'\Fcal=L\cdot L',
 LL'\Ecal=L\Ecal\Fcal=0,
 \Fcal^2=0,$
and the product of three classes pulled back from $T$ is zero.
This implies \eqref{eq:cubic}.\end{proof}

The divisor $\Ecal$ of $Y$ is characterized by the following

\begin{lemma}\label{lem:Eunique}
    \(\mathcal{E}\) is the unique prime divisor \(D \subset Y\) satisfying \(D^2\mathcal{F} = -1\) and \(D^3 > 0\).
\end{lemma}

\begin{proof}
    By \eqref{eq:familyintersections}, $\Ecal$ satisfies the required properties. Next we prove the uniqueness. Suppose $D\neq \Ecal$ be a prime divisor of $Y$ satisfying \(D^2\mathcal{F} = -1\) and \(D^3 > 0\). Then $B:=\beta(D)$ is a prime divisor in $Z=C\times T$ and $D$ is the proper transform of $B$. Using the decomposition $\Pic(C\times T)\cong\Pic(C)\oplus \Pic(T)$, we may write
\begin{equation}\label{eq:Bdecomp}
    B \sim  b \mathrm{pr}_1^*[q_0]+ \mathrm{pr}_2^* L,
\end{equation}
 where $L\in \Pic(T)$, $b\in \ZZ$, and $q_0\in C$ is a closed point. Then \[D \sim L - m\mathcal{E} + b \mathcal{F},\] where $m\ge 0$ is the multiplicity of $B$ along the blowup center $\Gamma_C$. We may choose a general point $u\in T$ such that $C\times \{u\}$ is not contained in $B$. Then by \eqref{eq:Bdecomp}, $b={\rm deg}(\mathcal{O}_Z(B)|_{C\times \{u\}})\ge 0$. 
 
 We claim that the intersection number $L.C\ge 0$. Consider the prime divisor
\[
C \times C \subset C \times T,
\]
where the second \(C\) denotes its image in \(T\) under $i$. First,
\(
B \neq C \times C.
\)
Indeed, if equality held, \(D\) would be the strict transform of \(C \times C\), so
\(
D \sim C - \mathcal{E}
\) since the multiplicity along the diagonal graph is one.  Then we have $L=C$ and 
\[
D^2\mathcal{F} =(L-\Ecal)^2 \Fcal=L^2 \Fcal-2L\Ecal\Fcal+\Ecal^2\Fcal= C^2 - 1 = -5,
\]
contradicting to $D^2\Fcal=-1$. So $B\neq C\times C$. Hence for a general point \(q \in C\), the curve
\(
\{q\} \times C
\)
is not contained in \(B\). Restricting \(B\) to this curve therefore gives an effective divisor whose degree is
\(
L \cdot C.
\)
Thus we have \(L \cdot C \geq 0\). This proves the claim.

By $D^2\Fcal=-1$, we have \[
-1=(L-m\Ecal+b\Fcal)^2\Fcal=L^2 - m^2.
\] Substitute \(s = -m\) and \(t = b\) into the cubic form \eqref{eq:cubic}, we get by $m \ge 0$ and $b \ge 0$
$$ 0< D^3= 3b(L^2 - m^2) - 3m^2(L \cdot C) - 2m^3= -3b - 3m^2(L \cdot C) - 2m^3 \le 0,$$
a contradiction. This completes the proof of the lemma.
\end{proof}

\begin{lemma}\label{lem:relative}
Every automorphism of $Y$ over $C$ preserves $\Ecal$.
Moreover, lifting induces an isomorphism
\[
 \Ine(T,C)\simeq\Aut_C(Y).
\]
\end{lemma}

\begin{proof}
Let $f\in\Aut_C(Y)$. Since $f$ preserves the class $\Fcal$ and the cubic intersection form on $\Pic(Y)$, the prime divisor $f(\Ecal)=\Ecal$ by Lemma \ref{lem:Eunique}. Let $Z'$ be the (reduced) image of the morphism
\[
 Y\longrightarrow Z\times Z,\qquad y\mapsto (\beta(y),\beta(f(y))),
\]
and let $r_1,r_2:Z'\to Z$ be the two projections. Clearly the
first projection is projective and birational. Moreover, for any $z\in Z$, the preimage $r_1^{-1}(z)$ consists of only one point since $f\in\Aut_C(Y)$ and $f(\Ecal)=\Ecal$. Thus $r_1$ is quasi-finite, which implies $r_1$ is a finite map since it is projective. Then $r_1$ is an isomorphism since it is birational and $Z$ is normal. Note that the morphism
\(
 g:=r_2\circ r_1^{-1}:Z\longrightarrow Z
\)
is an automorphism of $Z$ satisfying $g\beta=\beta f$. Since $g$ fixes $\Gamma_C$ pointwise, by \cite[Lemma 3.1]{DOY2}, we infer that $g={\rm id}_C\times f'$ for some $f'\in {\rm Ine}(T,C)$. This implies that lifting induces $\Ine(T,C)\simeq\Aut_C(Y)$. \end{proof}

\section{Proof of Theorem~\ref{thm:main}}\label{sec:proof}

Let $\eta$ be the generic point of $C$, write $k:=\CC(C)=\CC(t)$
with $t=x(\eta)$, and fix an algebraic closure $\bar k$ of $k$. From $p=\operatorname{pr}_1\circ\beta:Y\longrightarrow C$ and base change, we obtain
 $$X_\eta:= Y\times_C \Spec k\qquad \text{and} \qquad X_{\bar \eta}:= Y\times_C \Spec \bar k.$$
Note that the $\Gamma_C$ can be viewed as a section of the projection ${\rm pr}_1: Z=C\times T\rightarrow C$, and this gives a $k$-rational point $q_\eta$ of $T_k:=T\times_{\Spec\CC} \Spec k=Z\times_C \Spec k$. Since blowing up commutes with flat base change, we have $$X_\eta=\Bl_{q_\eta} T_k \qquad\text{and}\qquad X_{\bar\eta}=\Bl_{q_{\eta}}  T_{\bar k}.$$ Consider the group
\[
 G:=\Aut_{\bar k}(X_{\bar\eta}).
\]
Every element of $G$ is defined over $k$. Indeed, the geometric
Picard group $\Pic(X_{\bar\eta})$ is generated by the classes in \eqref{eq:pic}, all
defined over $k$. Thus an automorphism and each of its Galois
conjugates have the same Picard action. They coincide by
Lemma~\ref{lem:faithful}. Each automorphism is defined over a finite
extension of $k$, so Galois descent gives
\begin{equation}\label{eq:descent}
 G=\Aut_k(X_\eta).
\end{equation}

By spreading out over the curve $C$, for any $g\in G$, there exists a largest Zariski open dense subset $U_g\subset C$ such that $g$ extends to an automorphism of $p^{-1}(U_g)$ over $U_g$, and  we define $\mathcal B(g)$ to be the complement of $U_g$ in $C$. Then for every $g\in G$, $\mathcal B(g)$ is a finite set.  We have
\begin{equation}\label{eq:bad}
 \mathcal B(gh)\subset\mathcal B(g)\cup\mathcal B(h),\qquad
 \mathcal B(g^{-1})=\mathcal B(g).
\end{equation}

\begin{proposition}\label{prop:generic}
The group $G$ is not finitely generated.
\end{proposition}

\begin{proof}
Suppose the proposition is not true. Formula~\eqref{eq:bad} implies that
\[
 \mathcal B:=\bigcup_{g\in G}\mathcal B(g)
\]
is finite, since it is contained in the union of the bad sets of a
finite generating set.

For $n\in\Aut(T,C)$, the map
$(q,t)\mapsto(n(q),n(t))$ preserves $\Gamma_C$ and lifts to an
automorphism $\widetilde n$ of $Y$ covering $n|_C$. It acts on $G$
by conjugation, by \eqref{eq:descent} and
\[
 \mathcal B(\widetilde n g\widetilde n^{-1})
       =n|_C\bigl(\mathcal B(g)\bigr).
\]
The conjugate is a rational map over $C$, although
$\widetilde n$ itself may act nontrivially on the base.
Thus $\mathcal B$ is invariant under the descended maps ${(f_1)}_T,(\iota_0\iota_1)_T,\phi_T$.
Lemma~\ref{lem:maps} gives $\mathcal B=\varnothing$.

All elements of $G$ therefore extend over $C$. The extensions are
unique, so Lemma~\ref{lem:relative} gives
\[
 G=\Aut_C(Y)\simeq\Ine(T,C).
\]
This contradicts Proposition~\ref{prop:inertia}.
\end{proof}

\begin{proof}[Proof of Theorem~\ref{thm:main}]
Let $q\in C$ with $\tau=x(q)$ being transcendental over $\QQ$.
The algebraically closed fields $\bar k=\overline{\CC(t)}$
and $\CC$ have the same infinite transcendence degree over $\QQ$.
Choose transcendence bases containing $t$ and $\tau$, respectively,
and a bijection sending $t$ to $\tau$. We extend the resulting
isomorphism to their algebraic closures and get a field isomorphism
\[
 \sigma:\bar k\longrightarrow\CC,\qquad
 \sigma|_{\QQ}=\id,\qquad \sigma(t)=\tau.
\]
The family $Y\to C$, including the coordinate $x$, is defined over
$\QQ$. Transport of scalars by $\sigma$ consequently identifies
$X_{\bar\eta}$ with $X_q$ and therefore induces an isomorphism of abstract
groups
\[
 G\simeq\Aut(X_q).
\]
This use of $\sigma$ does not require it to fix the original copy
of $\CC$ in $\bar k$. By Proposition~\ref{prop:generic}, $\Aut(X_q)$ is not finitely generated. By Lemma~\ref{lem:faithful}, $\Aut(X_q)$ is discrete. This completes the proof.
\end{proof}

\begin{remark}
The proof uses descent of generic automorphisms to $T$ only after
assuming finite generation of $G$. The final change of ground field
gives an abstract group isomorphism; it does not assert that every
generic automorphism extends to the chosen complex fiber.
\end{remark}

\end{document}